\documentclass{gtart_a}

\title{Pro--$p$ groups and towers of rational homology spheres}

\author{Nigel Boston}
\givenname{Nigel}
\surname{Boston}
\address{Department of Mathematics\\
University of Wisconsin\\\newline
Van Vleck Hall\\
480 Lincoln Drive\\
Madison WI 53706\\
USA}
\email{boston@math.wisc.edu}
\urladdr{}

\author{Jordan S Ellenberg}
\givenname{Jordan S}
\surname{Ellenberg}
\email{ellenber@math.wisc.edu}
\urladdr{}

\volumenumber{10}
\issuenumber{}
\publicationyear{2006}
\papernumber{9}
\lognumber{0686}
\startpage{331}
\endpage{334}

\doi{10.2140/gt.2006.10.331}
\MR{}
\Zbl{}

\arxivreference{}   
\arxivpassword{}   

\keyword{pro--$p$ group}
\keyword{hyperbolic 3--manifold}
\keyword{rational homology sphere}
\subject{primary}{msc2000}{20E18}
\subject{secondary}{msc2000}{22E40}

\received{22 November 2005}
\revised{11 December 2005}
\accepted{2 January 2006}
\published{2 April 2006}
\publishedonline{2 April 2006}
\proposed{Walter Neumann}
\seconded{David Gabai, Tomasz Mrowka}
\corresponding{}
\editor{}
\version{}

\let\xysavmatrix\xymatrix
\def\xymatrix{\disablesubscriptcorrection\xysavmatrix}
\AtBeginDocument{\let\bar\wbar}

\DeclareMathOperator{\SL}{SL}

\newcommand{\field}[1]{\mathbb{#1}}

\newcommand{\F}{\field{F}}

\newcommand{\ra}{\to}

\newcommand{\ic}[1]{\mathfrak{#1}}

\newcommand{\HH}{\mathcal{H}}

\newcommand{\inj}{\hookrightarrow}

\newcommand{\beq}{\begin{displaymath}}
\newcommand{\eeq}{\end{displaymath}}
\newcommand{\beqn}{\begin{equation}}
\newcommand{\eeqn}{\end{equation}}

\makeatletter
\def\cnewtheorem#1[#2]#3{\newtheorem{#1}{#3}
\expandafter\let\csname c@#1\endcsname\c@thm}

\theoremstyle{plain}

\cnewtheorem{prop}[thm]{Proposition}
\cnewtheorem{cor}[thm]{Corollary}
\cnewtheorem{lem}[thm]{Lemma}

\theoremstyle{definition}
\cnewtheorem{defn}[thm]{Definition}
\cnewtheorem{conj}[thm]{Conjecture}
\cnewtheorem{exmp}[thm]{Example}

\theoremstyle{remark}
\cnewtheorem{rem}[thm]{Remark}
\cnewtheorem{rems}[thm]{Remarks}
\cnewtheorem{note}[thm]{Note}
\cnewtheorem{warn}[thm]{Warning}
\makeatother

\begin{document}

\begin{abstract}
In the preceding paper, Calegari and Dunfield exhibit a sequence of
hyperbolic 3--manifolds which have increasing injectivity radius, and
which, subject to some conjectures in number theory, are rational
homology spheres.  We prove unconditionally that these manifolds are
rational homology spheres, and give a sufficient condition for a tower
of hyperbolic 3--manifolds to have first Betti number 0 at each
level.  The methods involved are purely pro--p group theoretical.
\end{abstract}

\maketitle

In \cite{cale:dc}, Calegari and Dunfield give a conditional answer to
a question of Cooper \cite[Problem 3.58]{kirb:problems} by exhibiting
a series of hyperbolic $3$--manifolds $M_1, M_2, \ldots$, such that
\begin{itemize}
\item the injectivity radius of $M_n$ is unbounded;
\item subject to the Generalized Riemann Hypothesis and Langlands-type
conjectures about the existence of Galois representations attached to
automorphic forms, $H^1(M_n,\Q) = 0$ for all $n$.
\end{itemize}

These 3--manifolds are constructed as quotients of hyperbolic
$3$--space by certain arithmetic lattices in $\SL_2(\C)$.  In the
following note, we explain how to prove unconditionally that
$H^1(M_n,\Q) = 0$ for all $n$, without use of automorphic forms.  The
argument uses only the theory of pro--$\!p$ groups (see Dixon--du
Sautoy--Mann--Segal~\cite{ddms}). More specifically we prove that the
pro--$\!p$ completion of $\pi_1(M_n)$ is $p$--adic analytic. This should
generalize to some other lattices in $\SL_2(\C)$.  We emphasize,
however, that the present argument is not in {\em general} a
replacement for the argument of Calegari and Dunfield; we expect there
will be many hyperbolic manifolds to which the method of Galois
representations might be applicable, but whose fundamental groups do
not have analytic pro--$\!p$ completion.  In particular, it follows from
results of Lubotzky \cite[Theorem 1.2, Remark 1.4]{lubo:annals83} that
when $\Gamma$ is a lattice with $\dim H^1(\Gamma,\F_p) \geq 4$, the
pro--$\!p$ completion of $\Gamma$ is never analytic.  On the other hand,
the argument here does apply to some non-arithmetic lattices
\cite[Section 6.7]{cale:dc}.

{\bf Note}\qua We use number theorists' notation throughout, in which $\Z_3$
denotes the ring of $3$--adic integers, not the field with $3$ elements.


We recall some basic facts about cocompact lattices in $\SL_2(\C)$.
Let $\Gamma$ be a torsion-free cocompact lattice.  Then there is a number field
$K$ (which can be taken to be the {\em trace field} of $\Gamma$) 
and a quaternion algebra $A$ admitting an injection $\Gamma \inj A^\times$.
(See Maclachlan--Reid \cite[3.2]{macl:mare}.) For each prime $\ic{p}$
of $K$, let
$A_{\ic{p}}/K_{\ic{p}}$ be the completion of $A$ at the prime
$\ic{p}$, and write $A_{\ic{p}}^1$ for the sugroup of elements of norm
$1$. If $U$ is a uniformly powerful subgroup of $A_{\ic{p}}^\times$,
the lower $p$--central series $U=U_0,U_1,U_2,\ldots$ is defined by
$U_{i+1} = U_i^p[U,U_i]$.  Write $\HH$ for hyperbolic $3$--space; then
$\HH/\Gamma$ is a compact hyperbolic $3$--manifold, which is a rational
homology sphere just when $H^1(\Gamma,\Q) = 0$.

\begin{prop} Let $\Gamma$ be a cocompact lattice of $SL_2(\C)$ and let $\ic{p}$
  be a prime of $K$ such that
\begin{itemize}
\item the norm of $\ic{p}$ is an odd rational prime $p$;
\item the closure of the image of $\pi\co \Gamma \inj
  A_{\ic{p}}^\times$ contains an open pro--$\!p$ subgroup $U$ of
  $A_{\ic{p}}^1$ such that if $\Gamma_0 = \pi^{-1}(U)$, then
  $\Gamma_0/ \Gamma_0^p$ is isomorphic to $(\Z/p\Z)^3$.
\end{itemize}
Then every normal subgroup $H$ of $\Gamma_0$ with $p$--group quotient has
$H^1(H,\Q) = 0$.  In particular, taking $\Gamma_i$ to be $\pi^{-1}(U_i)$,
the tower of compact $3$--manifolds $\HH/\Gamma_i \,(i=0,1,\ldots)$ has
unbounded injectivity radius, and each $\HH/\Gamma_i$ is a rational homology
$3$--sphere.
\label{pr:main}
\end{prop}

\begin{proof}
The unboundedness of the injectivity radii of $\HH / \Gamma_i$ follows
immediately from the fact that the $\Gamma_i$ have trivial intersection.

  Write $T$ for the pro--$\!p$ completion of $\Gamma_0$.  By the {\em
  dimension} of a pro--$\!p$ group $T$, we mean the $\F_p$--dimension of
  $H^1(T_0,\F_p)$ for any uniformly powerful open subgroup $T_0$ of
  $T$ as in Dixon--du Sautoy--Mann--Segal \cite[Definition 4.7]{ddms}.
  The fact that $T/T^p
  \cong \Gamma_0/\Gamma_0^p \cong (\Z/p\Z)^3$ implies that $T$
  is powerful \cite[Definition 3.1(i)]{ddms} and has dimension at most $3$
  \cite[Theorem 3.8]{ddms}.  Since $U$ is torsion-free and has $U/U^p
  \cong (\Z/p\Z)^3$, it is uniformly powerful \cite[Theorem 4.5]{ddms} and
  has dimension $3$.   Since dimension is additive in exact
  sequences of pro--$\!p$ groups \cite[Theorem 4.8]{ddms}) we have that
  the surjection $T \ra U$ has finite kernel.  It is clear that every
  open subgroup of $U$ has finite abelianization; the same now follows
  for $T$.  This completes the proof.
\end{proof}

We now explain how to show that the tower of manifolds studied in
the preceding article in this volume, by Calegari and Dunfield
\cite{cale:dc} satisfies the conditions of \fullref{pr:main}.
We recall some definitions and notation from \cite{cale:dc}.  Let $D$ be
the quaternion algebra over $\Q(\sqrt{-2})$ which is ramified precisely
at the two primes $\pi$ and $\bar{\pi}$ dividing $3$, let $B$ be a
maximal order of $D$, and let $B^\times$ be the group of units of $B$.
Calegari and Dunfield consider a manifold $M_0$ whose fundamental group
is isomorphic to $B^\times / \pm 1$.

Let $B_\pi$ be the maximal order in the completion of $D$ at $\pi$; then
$B_\pi^\times$ is a profinite group with a finite-index pro--3 subgroup, and
the natural map  $B^\times \ra B_\pi^\times$ is an inclusion whose image contains a dense subgroup of the group $B_\pi^1$ of elements of reduced norm~$1$. 

Let $Q$ be the unique maximal two-sided ideal of $B_\pi$; then $(1+Q^n) \cap B_\pi^1$
is an open subgroup of $B_\pi^1$ for all $n \geq 0$, and is a pro--3 group
for $n \geq 1$.  Let $\Gamma_n$ be the preimage of $(1+Q^{n}) \cap B_\pi^1$ under $B^\times \ra
B_\pi^\times$.  Then the content of \cite[Theorem 1.4]{cale:dc} is that $\Gamma_n$
has finite abelianization for all sufficiently large $n$.  In a slight
discord of notation, the group denoted $\Gamma_2$ by Calegari and Dunfield
plays the role of $\Gamma_0$ in \fullref{pr:main}.   It remains only to
check that $\Gamma_2 / \Gamma_2^3$ is isomorphic to $(\Z/3\Z)^3$.  

A presentation of $\Gamma_1$ is obtained by using Magma to calculate the normal subgroups of index $4$ in $\pi_1(M_0)$, for which a Wirtinger presentation was given in \cite{cale:dc}.
\begin{verbatim}
Gamma1 := Group < a,b,c,d | a*b^{-1}*c^{-1}*b*a^{-1}*d*c*d^{-1},
                     a*b*a^{-1}*d*c*d*a^{-2}*b*c,
                     a*d*c*d*a^{-1}*b*d^{-2}*c^{-1}*b^{-1},
                     c*d^2*c*d^2*c*d^2, c^3, a*c*b*c*a*b*d^{-2} >
\end{verbatim}
Then $\Gamma_2$ is the kernel of the map from $\Gamma_1$ to its maximal
elementary abelian $3$--quotient.  One can easily compute a presentation
of $\Gamma_2$ (too long to be worth including here) and from there it is a
simple matter to compute $\Gamma_2 / \Gamma_2^3$.  We have thus shown that the
manifolds appearing in \cite{cale:dc} are all rational homology
spheres.

\begin{rem} The group $\Gamma_2$ does not have the congruence subgroup
  property (see Lubotzky \cite{lubo:annals83}); however, one might think of
  \fullref{pr:main} as asserting a kind of ``pro--3 congruence
  subgroup property'': every finite-index normal subgroup of $\Gamma_2$ whose
  quotient is a $3$--group is indeed congruence.  It would be
  interesting to understand which lattices in $\SL_2(\C)$ are residual
  $p$--groups with the pro--$\!p$ congruence subgroup property for some $p$.
  This property certainly cannot hold for {\em  all} lattices, since
  there exist lattices with infinite
  abelianization (see Labesse--Schwermer \cite{labe:ls} and Lubotzky
  \cite{lubo:annals97}). For such
  lattices, $\dim H^1(\Gamma,\F_p) \leq 3$ by Lubotzky \cite{lubo:annals83}.
\end{rem}

\subsection*{Acknowledgments}
The second author was partially supported by NSF-{\sc career} Grant DMS-0448750
and a Sloan Research Fellowship.  The authors thank the anonymous referee
for his helpful comments and simplification of the proof of their main
proposition.

\bibliographystyle{gtart}
\bibliography{link}

\end{document}